\documentclass[12pt]{amsart}
\usepackage{amscd,amsmath,amssymb,enumerate,amsfonts,mathrsfs,txfonts}
\usepackage{comment}
\usepackage{hyperref}
\usepackage{xcolor}
\usepackage{tikz,pgfplots,pgf}
\usepackage[all]{xy}

\tikzset{node distance=3cm, auto}
\usetikzlibrary{calc} 
\usetikzlibrary{trees} 

\usepackage{vmargin}
\setpapersize{A4}
\setmargins{2.5cm}      
{1.5cm}                        
{16.5cm}                      
{23.42cm}                   
{10pt}                          
{1cm}                           
{0pt}                          
{2cm}
\newtheorem{theorem}{Theorem}[section]

\newtheorem{proposition}[theorem]{Proposition}
\newtheorem{corollary}[theorem]{Corollary}
\theoremstyle{definition}
\newtheorem{definition}[theorem]{Definition}

\numberwithin{equation}{section}

\def\F{\mathcal{F}}
\def\G{\mathcal{G}}
\def\H{\mathcal{H}}
\def\I{\mathcal{I}}
\def\K{\mathcal{K}}

\def\L{\mathcal{L}}

\def\Nu{\mathcal{N}}
\def\QN{\mathcal{QN}}
\def\Q{\mathcal{Q}}

\def\D{\mathbb{D}}
\def\abco{\mathrm{abco}}
\def\conv{\mathrm{conv}}

\def\weak{\mathrm{weak}}

\makeatletter
\@namedef{subjclassname@2020}{\textup{2020} Mathematics Subject Classification}
\makeatother

\begin{document}

\title[On holomorphic mappings with relatively $p$-compact range]{On holomorphic mappings with relatively $p$-compact range}

\author[A. Jim\'enez-Vargas]{A. Jim\'enez-Vargas}
\address[A. Jim\'enez-Vargas]{Departamento de Matem\'aticas, Universidad de Almer\'ia, 04120, Almer\'ia, Spain}
\email{ajimenez@ual.es}

\date{\today}



\begin{abstract}
Related to the concept of $p$-compact operator with $p\in [1,\infty]$ introduced by Sinha and Karn \cite{SinKar-02}, this paper deals with the space $\H^\infty_{\K_p}(U,F)$ of all Banach-valued holomorphic mappings on an open subset $U$ of a complex Banach space $E$ whose ranges are relatively $p$-compact subsets of $F$. We characterize such holomorphic mappings as those whose Mujica's linearisations on the canonical predual of $\H^\infty(U)$ are $p$-compact operators. This fact allows us to make a complete study of them. We show that $\H^\infty_{\K_p}$ is a surjective Banach ideal of bounded holomorphic mappings which is generated by composition with the ideal of $p$-compact operators and contains the Banach ideal of all right $p$-nuclear holomorphic mappings. We also characterize holomorphic mappings with relatively $p$-compact ranges as those bounded holomorphic mappings which factorize through a quotient space of $\ell_{p^*}$ or as those whose transposes are quasi $p$-nuclear operators (respectively, factor through a closed subspace of $\ell_p$). 
\end{abstract}
\maketitle


\section*{Introduction}\label{section 0}
Inspired by classical Grothendieck's characterization of a relatively compact subset of a Banach space as a subset of the convex hull of a norm null sequence of vectors \cite{Gro-55}, Sinha and Karn \cite{SinKar-02} introduced and studied $p$-compact sets and $p$-compact operators with $p\in [1,\infty]$.

Let $E$ be a complex Banach space with closed unit ball $B_E$. Let $p\in (1,\infty)$ and let $p^*$ denote the \textit{conjugate index of $p$} given by $p^*=p/(p-1)$. Given $p\in [1,\infty)$, $\ell_p(E)$ denotes the Banach space of all absolutely $p$-summable sequences $(x_n)$ in $E$, equipped with the norm
$$
\left\|(x_n)\right\|_p=\left(\sum_{n=1}^\infty\left\|x_n\right\|^p\right)^{1/p},
$$
and $c_0(E)$ stands for the Banach space of all norm null sequences $(x_n)$ in $E$, endowed with the norm
$$
\left\|(x_n)\right\|_\infty=\sup\left\{\left\|x_n\right\|\colon n\in\mathbb{N}\right\}.
$$
In the case $E=\mathbb{C}$, we will simply write $\ell_p$ and $c_0$. 

For $p\in (1,\infty)$, the \textit{$p$-convex hull} of a sequence $(x_n)\in\ell_p(E)$ is defined by
$$
p\text{-}\conv(x_n)=\left\{\sum_{n=1}^\infty a_nx_n\colon (a_n)\in B_{\ell_{p^*}}\right\}.
$$
Similarly, we set 
\begin{align*}
1\text{-}\conv(x_n)&=\left\{\sum_{n=1}^\infty a_nx_n\colon (a_n)\in B_{c_0}\right\},\qquad (x_n)\in \ell_1(E),\\
\infty\text{-}\conv(x_n)&=\left\{\sum_{n=1}^\infty a_nx_n\colon (a_n)\in B_{\ell_1}\right\},\qquad (x_n)\in c_0(E).
\end{align*}
Note that $\infty\text{-}\conv(x_n)$ coincides with $\overline{\abco}(\{x_n\colon n\in\mathbb{N}\})$, the norm-closed absolutely convex hull of the set $\{x_n\colon n\in\mathbb{N}\}$ in $E$.

Following \cite{SinKar-02}, given $p\in [1,\infty]$, a set $K\subseteq E$ is said to be \textit{relatively $p$-compact} if there is a sequence $(x_n)\in\ell_p(E)$ ($(x_n)\in c_0(E)$ if $p=\infty$) such that $K\subseteq p\text{-}\conv(x_n)$. Such a sequence is not unique but a \textit{measure of the size of the $p$-compact set $K$} is introduced in \cite[p. 297]{DelPiSer-10} (see also \cite{LasTur-12}) defining
$$
m_p(K,E)=
\begin{cases}
\inf\{\left\|(x_n)\right\|_p\colon (x_n)\in\ell_p(E),\, K\subseteq p\text{-}\conv(x_n)\}        &\text{if } 1\leq p<\infty,\\
\inf\{\left\|(x_n)\right\|_p\colon (x_n)\in c_0(E),\, K\subseteq p\text{-}\conv(x_n)\}     &\text{if } p=\infty,
\end{cases}
$$
and $m_p(K,E)=\infty$ if $K$ is not $p$-compact. We frequently will write $m_p(K)$ instead of $m_p(K,E)$.

A linear operator between Banach spaces $T\colon E\to F$ is said to be \textit{$p$-compact} if $T(B_E)$ is a relatively $p$-compact subset of $F$. The space of all $p$-compact linear operators from $E$ to $F$ is denoted by $\K_p(E,F)$, and $\K_p$ is a Banach operator ideal endowed with the norm $k_p(T)=m_p(T(B_E))$ (see \cite[Theorem 4.2]{SinKar-02} and \cite[Proposition 3.15]{DelPiSer-10}).

From the cited Grothendieck's result \cite{Gro-55}, a set $K\subseteq E$ is relatively compact if and only if for every $\varepsilon>0$, there is a sequence $(x_n)\in c_0(E)$ with $\left\|(x_n)\right\|_\infty\leq \sup_{x\in K}\left\|x\right\|+\varepsilon$ such that $K\subseteq\infty\text{-}\conv(x_n)$. Hence we can consider compact sets as $\infty$-compact sets and, in this way, compact operators can be viewed as $\infty$-compact operators with $k_\infty$ being the usual operator norm. 

The work of Sinha and Karn \cite{SinKar-02} motivated many papers on $p$-compactness in operator spaces (see \cite{AinLilOja-12,ChaDimGal-19,ChoKim-10,DelPiSer-10,GalLasTur-12,LasTur-12,Pie-14}, among others) and also in Lipschitz spaces \cite{AchDahTur-19,AchDahTur-20}.

The extension of this theory to the polynomial and holomorphic settings was addressed in \cite{AroCalGarMae-16,AroMaeRue-10}. In these environments, the property of $p$-compactness was studied from the following local point of view: a mapping $f\colon U\to F$ is said to be \textit{locally $p$-compact} if every point $x\in U$ has a neighborhood $U_x\subseteq U$ such that $f(U_x)$ is relatively $p$-compact in $F$. 

The aim of this note is to study a subclass of locally $p$-compact holomorphic mappings, namely, \textit{holomorphic mappings with relatively $p$-compact range}. Notice that every such mapping is locally $p$-compact but the converse is not true. For instance, if $\D$ denotes the open complex unit disc, the holomorphic mapping $f\colon\D\to c_0$ defined by $f(z)=\sum_{n=1}^\infty z^ne_n$, where $(e_n)$ is the canonical basis of $\ell_1$, is locally $1$-compact but it has not relatively compact range (see \cite[Example 3.2]{Muj-91}) and, consequently, neither relatively $p$-compact range for any $p\geq 1$. 

Our motivation to deal with this class of mappings also arises from the study (initiated in \cite{Muj-91} and continued in \cite{CabJimRuiSep-22,JimRuiSep-22}) on the Banach space $\H^\infty_{\K}(U,F)$ formed by all holomorphic mappings from $U$ to $F$ with relatively compact range, equipped with the supremum norm.

We now briefly describe the content of this paper. Let $E$ and $F$ be complex Banach spaces, $U$ an open subset of $E$ and $p\in [1,\infty]$. Let $\H^\infty(U,F)$ denote the Banach space of all bounded holomorphic mappings from $U$ into $F$, endowed with the supremum norm. In particular, $\H^\infty(U)$ stands for $\H^\infty(U,\mathbb{C})$. 

In \cite{Muj-91}, Mujica provided a linearisation method of the members of $\H^\infty(U,F)$, which will be an essential tool in our analysis of the subject. If $\G^\infty(U)$ is the canonical predual of $\H^\infty(U)$ obtained by Mujica \cite{Muj-91} via an identification denoted $g_U$, we will establish that a bounded holomorphic mapping $f\colon U\to F$ has relatively $p$-compact range if and only if Mujica's linearisation $T_f\colon\G^\infty(U)\to F$ is a $p$-compact operator. This fact has some interesting applications. 

If $\H^\infty_{\K_p}(U,F)$ denotes the space of all holomorphic mappings with relatively $p$-compact range $f\colon U\to F$  with the natural norm $k^{\H^\infty}_p(f)=m_p(f(U))$, we will prove that $\H^\infty_{\K_p}$ is a surjective Banach ideal of bounded holomorphic mappings which is generated by composition with the ideal $\K_p$ of $p$-compact operators. This means that each mapping $f\in\H^\infty_{\K_p}(U,F)$ admits a factorization $f=T\circ g$, where $G$ is a complex Banach space, $g\in\H^\infty(U,G)$ and $T\in\K_p(G,F)$. Moreover, $k^{\H^\infty}_p(f)$ coincides with $\inf\{k_p(T)\left\|g\right\|_\infty\}$, where the infimum is extended over all such factorizations of $f$ and, curiously, this infimum is attained at the factorization $f=T_f\circ g_U$ due to Mujica \cite{Muj-91}.


In parallelism with the linear case, we introduce the notion of right $p$-nuclear holomorphic mapping from $U$ to $F$, study its linearisation on $\G^\infty(U)$, analyse its ideal property and show that every right $p$-nuclear holomorphic mapping has relatively $p$-compact range.

Moreover, we characterize the members of $\H^\infty_{\K_p}(U,F)$ as those bounded holomorphic mappings from $U$ to $F$ which factorize through a quotient space of $\ell_{p^*}$, and also as those whose transposes are quasi $p$-nuclear operators (respectively, factor through a closed subspace of $\ell_p$).


\section{The results}\label{section 2}

From now on, unless otherwise stated, $E$ and $F$ will denote complex Banach spaces, $U$ will be an open subset of $E$ and $p\in [1,\infty]$. 

As usual, $\L(E,F)$ denotes the Banach space of all bounded linear operators from $E$ to $F$ endowed with the operator canonical norm, and $E^*$ stands for the dual space of $E$. The subspaces of $\L(E,F)$ consisting of all compact operators and finite-rank bounded operators from $E$ to $F$ will be denoted by $\K(E,F)$ and $\F(E,F)$, respectively. 

In this section, we will study holomorphic mappings $f\colon U\to F$ so that $f(U)$ is a relatively $p$-compact subset of $F$. We denote by $\H^\infty_{\K_p}(U,F)$ the set formed by such mappings and, for each $f\in\H^\infty_{\K_p}(U,F)$, we define $k_p^{\H^\infty}(f)=m_p(f(U))$.

The space of all holomorphic mappings with relatively compact range from $U$ to $F$, denoted $\H^\infty_{\K}(U,F)$, is a Banach space equipped with the supremum norm (see \cite[Corollary 2.11]{JimRuiSep-22}). On account of the following result, we will only study in this paper the case $1\leq p<\infty$. 

\begin{proposition}\label{new}
$\H^\infty_{\K_\infty}(U,F)=\H^\infty_{\K}(U,F)$ and $k_\infty^{\H^\infty}(f)=\left\|f\right\|_\infty$ for all $\H^\infty_{\K_\infty}(U,F)$.
\end{proposition}

\begin{proof}
Let $f\in\H^\infty_{\K_\infty}(U,F)$ and let $(y_n)\in c_0(F)$ be such that $f(U)\subseteq\infty\text{-}\conv(y_n)$. Since $\infty\text{-}\conv(y_n)$ is relatively compact in $F$, it follows that $f\in\H^\infty_{\K}(U,F)$ 
with $\left\|f\right\|_\infty\leq \left\|(y_n)\right\|_\infty$, and taking infimum over all such sequences $(y_n)$, we have $\left\|f\right\|_\infty\leq k_\infty^{\H^\infty}(f)$. 

Conversely, let $f\in\H^\infty_{\K}(U,F)$. By classical Grothendieck's result, for every $\varepsilon>0$, there exists $(y_n)\in c_0(F)$ with $\left\|(y_n)\right\|_\infty\leq \left\|f\right\|_\infty+\varepsilon$ such that $f(U)\subseteq\infty\text{-}\conv(y_n)$. Hence $f\in\H^\infty_{\K_\infty}(U,F)$ and $k_\infty^{\H^\infty}(f)\leq\left\|f\right\|_\infty$. 
\end{proof}


\subsection{Linearisation}

Our first aim is to characterize holomorphic mappings with relatively $p$-compact range in terms of the $p$-compactness of their linearisations on the canonical predual of $\H^\infty(U)$. 

Towards this end, we first recall the following result due to Mujica \cite{Muj-91} concerning linearisation of holomorphic mappings on Banach spaces. 

\begin{theorem}\label{main-theo}\cite{Muj-91}
Let $E$ be a complex Banach space and $U$ be an open set in $E$. Let $\G^\infty(U)$ denote the norm-closed linear subspace of $\H^\infty(U)^*$ generated by the functionals $\delta(x)\in\H^\infty(U)^*$ with $x\in U$, defined by $\delta(x)(f)=f(x)$ for all $f\in\H^\infty(U)$. 
\begin{enumerate}
\item The mapping $g_U\colon U\to\G^\infty(U)$ defined by $g_U(x)=\delta(x)$ is holomorphic with $\left\|\delta(x)\right\|=1$ for all $x\in U$.
\item For every complex Banach space $F$ and every mapping $f\in\H^\infty(U,F)$, there exists a unique operator $T_f\in\L(\G^\infty(U),F)$ such that $T_f\circ g_U=f$. Furthermore, $\left\|T_f\right\|=\left\|f\right\|_\infty$.
\item For every complex Banach space $F$, the mapping $f\mapsto T_f$ is an isometric isomorphism from $\H^\infty(U,F)$ onto $\L(\G^\infty(U),F)$. 
\item $\H^\infty(U)$ is isometrically isomorphic to $\G^\infty(U)^*$, via the mapping $J_U\colon\H^\infty(U)\to\G^\infty(U)^*$ given by $J_U(f)(g_U(x))=f(x)$ for all $f\in\H^\infty(U)$ and $x\in U$.
\item $B_{\G^\infty(U)}$ coincides with $\overline{\abco}(g_U(U))$. $\hfill\Box$
\end{enumerate}
\end{theorem}

We are now ready to state the aforementioned characterization.

\begin{theorem}\label{linear}
Let $p\in [1,\infty)$ and $f\in\H^\infty(U,F)$. The following conditions are equivalent:
\begin{enumerate}
	\item $f$ has relatively $p$-compact range.
	\item $T_f\colon\G^\infty(U)\to F$ is a $p$-compact linear operator.
\end{enumerate}
In this case, $k_p^{\H^\infty}(f)=k_p(T_f)$. Furthermore, the mapping $f\mapsto T_f$ is an isometric isomorphism from $(\H^\infty_{\K_p}(U,F),k_p^{\H^\infty})$ onto $(\K_p(\G^\infty(U),F),k_p)$.
\end{theorem}

\begin{proof}
Using Theorem \ref{main-theo}, we have the following inclusions:
\begin{align*}
f(U)=T_f\circ g_U(U)&\subseteq T_f(\overline{\abco}(g_U(U)))=T_f(B_{\G^\infty(U)})\\
                    &\subseteq \overline{\abco}(T_f\circ g_U(U))=\overline{\abco}(f(U)).
\end{align*}

$(ii)\Rightarrow (i)$: If $T_f\in\K_p(\G^\infty(U),F)$, then $f\in\H^\infty_{\K_p}(U,F)$ with 
$$
k_p^{\H^\infty}(f)=m_p(f(U))\leq m_p(T_f(B_{\G^\infty(U)}))=k_p(T_f)
$$
by the first inclusion above. 

$(i)\Rightarrow (ii)$: If $f\in\H^\infty_{\K_p}(U,F)$, then $T_f\in\K_p(\G^\infty(U),F)$ with 
$$
k_p(T_f)=m_p(T_f(B_{\G^\infty(U)}))\leq m_p(\overline{\abco}(f(U)))=m_p(f(U))=k_p^{\H^\infty}(f),
$$
by the second inclusion above and the known fact that a set $K\subseteq F$ is $p$-compact if and only if $\overline{\abco}(K)$ is $p$-compact, in whose case $m_p(K)=m_p(\overline{\abco}(K))$.

The last assertion of the statement follows immediately from part (iii) of Theorem \ref{main-theo} and from the above proof. 
\end{proof}


\subsection{Banach ideal property}

Our next goal is to study the Banach ideal structure of $(\H^\infty_{\K_p},k_p^{\H^\infty})$. Inspired by the notion of Banach operator ideal \cite{Pie-80}, the following type of ideals was considered in \cite{CabJimRuiSep-22}.

An \textit{ideal of bounded holomorphic mappings} (or simply, a \textit{bounded-holomorphic ideal}) is a subclass $\I^{\H^\infty}$ of the class of bounded holomorphic mappings $\H^\infty$ such that for each complex Banach space $E$, each open subset $U$ of $E$ and each complex Banach space $F$, the components 
$$
\I^{\H^\infty}(U,F):=\I^{\H^\infty}\cap\H^\infty(U,F) 
$$
satisfy the following three conditions:

\begin{enumerate}
\item[(I1)] $\I^{\H^\infty}(U,F)$ is a linear subspace of $\H^\infty(U,F)$,
\item[(I2)] For any $g\in\H^\infty(U)$ and $y\in F$, the mapping $g\cdot y\colon x\mapsto g(x)y$ from $U$ to $F$ is in $\I^{\mathcal{H}^\infty}(U,F)$,
\item[(I3)] \textit{The ideal property}: if $H,G$ are complex Banach spaces, $V$ is an open subset of $H$, $h\in\H(V,U)$, $f\in\I^{\H^\infty}(U,F)$ and $S\in\L(F,G)$, then $S\circ f\circ h\in\I^{\H^\infty}(V,G)$.
\end{enumerate}

Suppose that a function $\left\|\cdot\right\|_{\I^{\H^\infty}}\colon\I^{\H^\infty}\to\mathbb{R}_0^+$ satisfies the following three properties:

\begin{enumerate}
\item[(N1)] $(\I^{\H^\infty}(U,F),\left\|\cdot\right\|_{\I^{\H^\infty}})$ is a normed (Banach) space with $\left\|f\right\|_\infty\leq\left\|f\right\|_{\I^{\H^\infty}}$ for all $f\in\I^{\H^\infty}(U,F)$, 
\item[(N2)] $\left\|g\cdot y\right\|_{\I^{\H^\infty}}=\left\|g\right\|_\infty\left\|y\right\|$ for all $g\in\H^\infty(U)$ and $y\in F$, 
\item[(N3)] If $H,G$ are complex Banach spaces, $V$ is an open subset of $H$, $h\in\H(V,U)$, $f\in\I^{\H^\infty}(U,F)$ and $S\in\L(F,G)$, then $\left\|S\circ f\circ h\right\|_{\I^{\H^\infty}}\leq \left\|S\right\|\left\|f\right\|_{\I^{\H^\infty}}$.
\end{enumerate}
Then $(\I^{\H^\infty},\left\|\cdot\right\|_{\I^{\H^\infty}})$ is called a \textit{normed (Banach) bounded-holomorphic ideal}.

A normed bounded-holomorphic ideal $\I^{\H^\infty}$ is said to be:
\begin{enumerate}
	\item[(R)] \textit{regular} if for any $f\in\H^\infty(U,F)$, we have that $f\in\I^{\H^\infty}(U,F)$ with $\left\|f\right\|_{\I^{\H^\infty}}=\left\|\kappa_F\circ f\right\|_{\I^{\H^\infty}}$ 	whenever $\kappa_F\circ f\in\I^{\H^\infty}(U,F^{**})$, where $\kappa_F$ denotes the isometric linear embedding from $F$ into $F^{**}$.
	\item[(S)] \textit{surjective} if for any mapping $f\in\H^\infty(U,F)$, any open subset $V$ of a complex Banach space $G$ and any surjective mapping $\pi\in\H(V,U)$, we have that $f\in\I^{\H^\infty}(U,F)$ with $\left\|f\right\|_{\I^{\H^\infty}}=\left\|f\circ \pi\right\|_{\I^{\H^\infty}}$ whenever $f\circ \pi\in\I^{\H^\infty}(V,F)$.
\end{enumerate}

Bearing in mind Theorem \ref{linear}, Theorem 3.2 in \cite{AroBotPelRue-10} (see also Theorem 2.4 in \cite{CabJimRuiSep-22}) shows that $\H^\infty_{\K_p}$ is generated by composition with the operator ideal $\K_p$ (see \cite[Definition 2.3]{CabJimRuiSep-22}).

\begin{corollary}\label{messi-3}
Let $p\in [1,\infty)$ and $f\in\H^\infty(U,F)$. The following conditions are equivalent:
\begin{enumerate}
\item $f\colon U\to F$ has relatively $p$-compact range. 
\item $f=T\circ g$ for some complex Banach space $G$, $g\in\H^\infty(U,G)$ and $T\in\K_p(G,F)$. 
\end{enumerate}
In this case, we have
$$
k_p^{\H^\infty}(f)=\left\|f\right\|_{\K_p\circ\H^\infty}:=\inf\{k_p(T)\left\|g\right\|_\infty\},
$$
where the infimum runs over all factorizations of $f$ as in $(ii)$, and this infimum is attained at $T_f\circ g_U$ (\textit{Mujica's factorization of $f$} \cite{Muj-91}). 

Furthermore, the mapping $f\mapsto T_f$ is an isometric isomorphism from $(\K_p\circ\H^\infty(U,F),\left\|\cdot\right\|_{\K_p\circ\H^\infty})$ onto $(\K_p(\G^\infty(U),F),k_p)$. $\hfill\qed$
\end{corollary}

The following result gathers some Banach ideal properties of $\H^\infty_{\K_p}$.

\begin{theorem}\label{ideal}
For each $p\in[1,\infty)$, $(\H^\infty_{\K_p},k_p^{\H^\infty})$ is a surjective Banach bounded-holomorphic ideal. Furthermore, the ideal $(\H^\infty_{\K_p}(U,F),k_p^{\H^\infty})$ is regular whenever $F$ is reflexive. 
\end{theorem}

\begin{proof}
In view of Corollary \ref{messi-3}, Corollary 2.5 in \cite{CabJimRuiSep-22} yields that $(\H^\infty_{\K_p},k_p^{\H^\infty})$ is a Banach bounded-holomorphic ideal. Then we only have to study its surjectivity and its regularity.

(S) Let $f\in\H^\infty(U,F)$ and assume that $f\circ \pi\in \H^\infty_{\K_p}(V,F)$, where $V$ is an open subset of a complex Banach space $G$ and $\pi\in\H(V,U)$ is surjective. Since $f(U)=(f\circ \pi)(V)$, it is immediate that $f\in\H^\infty_{\K_p}(U,F)$ with $k_p^{\H^\infty}(f)=k_p^{\H^\infty}(f\circ \pi)$. Hence $(\H^\infty_{\K_p},k_p^{\H^\infty})$ is surjective.

(R) Suppose now that $F$ is reflexive. Let $f\in\H^\infty(U,F)$ and assume that $\kappa_F\circ f\in\H^\infty_{\K_p}(U,F^{**})$. We can take a sequence $(y_n)$ in $\ell_p(F)$ (in $c_0(F)$ if $p=1$) such that $(\kappa_F\circ f)(U)\subseteq p\text{-}\conv(\kappa_F(y_n))$, that is, $\kappa_F(f(U))\subseteq \kappa_F(p\text{-}\conv(y_n))$ which yields $f(U)\subseteq p\text{-}\conv(y_n)$ by the injectivity of $\kappa_F$. Hence $f\in\H^\infty_{\K_p}(U,F)$ with $k_p^{\H^\infty}(f)\leq \left\|(y_n)\right\|_p=\left\|(\kappa_F(y_n))\right\|_p$, and so $k_p^{\H^\infty}(f)\leq k^{\H^\infty}_p(\kappa_F\circ f)$ by extending this infimum over all such sequences $(\kappa_F(y_n))$. The converse inequality follows from the condition (N3) satisfied by $(\H^\infty_{\K_p},k_p^{\H^\infty})$, and this completes the proof.
\end{proof}


\subsection{Factorization}

We now present a factorization result for holomorphic mappings with relatively $p$-compact range which should be compared with \cite[Proposition 2.9]{GalLasTur-12}. 

\begin{corollary}\label{messi-4}
Let $p\in [1,\infty)$ and $f\in\H^\infty(U,F)$. The following conditions are equivalent:
\begin{enumerate}
\item $f\colon U\to F$ has relatively $p$-compact range. 
\item There exist a closed subspace $M$ in $\ell_{p^*}$ ($c_0$ instead of $\ell_{p^*}$ if $p=1$), a separable Banach space $G$, an operator $T$ in $\K_p(\ell_{p^*}/M,G)$, a mapping $g$ in $\H^\infty_{\K}(U,\ell_{p^*}/M)$ and an operator $S$ in $\K(G,F)$ such that $f=S\circ T\circ g$. 
\end{enumerate}
In this case, $k_p^{\H^\infty}(f)=\inf\{\left\|S\right\|k_p(T)\left\|g\right\|_\infty\}$, where the infimum is extended over all factorizations of $f$ as in $(ii)$. 
\end{corollary}

\begin{proof}
We will only prove it for $p\in (1,\infty)$. The case $p=1$ is similarly obtained.

$(i)\Rightarrow (ii)$: Suppose that $f\in\H^\infty_{\K_p}(U,F)$. By Theorem \ref{linear}, $T_f\in\K_p(\G^\infty(U),F)$ with $k_p(T_f)=k_p^{\H^\infty}(f)$. Applying \cite[Proposition 2.9]{GalLasTur-12}, for each $\varepsilon>0$, there exist a closed subspace $M\subseteq\ell_{p^*}$,  
a separable Banach space $G$, an operator $T\in\K_p(\ell_{p^*}/M,G)$, an operator $S\in\K(G,F)$ and an operator $R\in\K(\G^\infty(U),\ell_{p^*}/M)$ such that $T_f=S\circ T\circ R$ with $\left\|S\right\|k_p(T)\left\|R\right\|\leq k_p(T_f)+\varepsilon$. Moreover, $R=T_g$ with $\left\|g\right\|_\infty=\left\|R\right\|$ for some $g\in\H^\infty_{\K}(U,\ell_{p^*}/M)$ by \cite[Corollary 2.11]{JimRuiSep-22}. Thus we obtain  
$$
f=T_f\circ g_U=S\circ T\circ R\circ g_U=S\circ T\circ T_g\circ g_U=S\circ T\circ g,
$$
with 
$$
\left\|S\right\|k_p(T)\left\|g\right\|_\infty=\left\|S\right\|k_p(T)\left\|R\right\|\leq k_p(T_f)+\varepsilon=k_p^{\H^\infty}(f)+\varepsilon.
$$
Since $\varepsilon$ was arbitrary, we deduce that $\left\|S\right\|k_p(T)\left\|g\right\|_\infty\leq k_p^{\H^\infty}(f)$.

$(ii)\Rightarrow (i)$: Assume that $f=S\circ T\circ g$ is a factorization as in $(ii)$. Since $S\circ T\in\K_p(\ell_{p^*}/M,F)$ by the ideal property of $\K_p$, Corollary \ref{messi-3} yields that $f\in\H^\infty_{\K_p}(U,F)$ with 
$$
k_p^{\H^\infty}(f)\leq k_p(S\circ T)\left\|g\right\|_\infty\leq \left\|S\right\|k_p(T)\left\|g\right\|_\infty,
$$ 
and taking infimum over all such factorizations of $f$, we have $k_p^{\H^\infty}(f)\leq\inf\{\left\|S\right\|k_p(T)\left\|g\right\|_\infty\}$. 
\end{proof}


\subsection{Transposition}

Let us recall that the \textit{transpose} of a mapping $f\in\H^\infty(U,F)$ is the bounded linear operator $f^t\colon F^*\to\H^\infty(U)$ defined by 
$$
f^t(y^*)=y^*\circ f\qquad (y^*\in F^*).
$$
Moreover, $||f^t||=\left\|f\right\|_\infty$ and $f^t=J_U^{-1}\circ(T_f)^*$, where $J_U\colon\H^\infty(U)\to\G^\infty(U)^*$ is the isometric isomorphism defined in Theorem \ref{main-theo}.

The Banach ideal $\K_p$ is associated by duality with the ideal of quasi-$p$-nuclear operators. According to \cite{PerPie-69}, for every $p\in [1,\infty)$, an operator $T\in\L(E,F)$ is said to be \textit{quasi $p$-nuclear} if there is a sequence $(x^*_n)\in\ell_{p}(E^*)$ such that
$$
\left\|T(x)\right\|\leq \left(\sum_{n=1}^\infty\left|x^*_n(x)\right|^p\right)^{1/p}\qquad (x\in E).
$$
If $\QN_p(E,F)$ denotes the set formed by such operators, then $\QN_p$ is a Banach operator ideal equipped with the norm
$$
\nu^{\Q}_p(T)=\inf\left\{\left\|(x^*_n)\right\|_p\colon \left\|T(x)\right\|\leq \left(\sum_{n=1}^\infty\left|x^*_n(x)\right|^p\right)^{1/p},\; \forall x\in E\right\}.
$$

By \cite[Proposition 3.8]{DelPiSer-10}, an operator $T\in\K_p(E,F)$ if and only if its adjoint $T^*\in\QN_p(F^*,E^*)$. Moreover, $k_p(T)=\nu^\Q_p(T^*)$ by \cite[Corollary 2.7]{GalLasTur-12}. A holomorphic version of this result can be stated as follows.

\begin{theorem}\label{now}
Let $p\in [1,\infty)$ and $f\in\H^\infty(U,F)$. The following conditions are equivalent:
\begin{enumerate}
\item $f\colon U\to F$ has relatively $p$-compact range.
\item $f^t\colon F^*\to\H^\infty(U)$ is a quasi $p$-nuclear operator. 
\end{enumerate}
In this case, $k_p^{\H^\infty}(f)=\nu_p^{\Q}(f^t)$.
\end{theorem}

\begin{proof}
Applying Theorem \ref{linear}, \cite[Corollary 2.7]{GalLasTur-12} and \cite[p. 32]{PerPie-69}, respectively, we have
\begin{align*}
f\in\H^\infty_{\K_p}(U,F)&\Leftrightarrow T_f\in\K_p(\G^\infty(U),F)\\
                             &\Leftrightarrow (T_f)^*\in\QN_p(F^*,\G^\infty(U)^*)\\
                             &\Leftrightarrow f^t=J_U^{-1}\circ(T_f)^*\in\QN_p(F^*,\H^\infty(U)).
\end{align*}
In this case, $k_p^{\H^\infty}(f)=k_p(T_f)=\nu^\Q_p((T_f)^*)=\nu^{\Q}_p(f^t)$.
\end{proof}

Given $p\in [1,\infty)$, let us recall (see \cite{Pie-80}) that an operator $T\in\mathcal{L}(E,F)$ is \textit{$p$-summing} if there exists a constant $C\geq 0$ such that, regardless of the natural number $n$ and regardless of the choice of vectors $x_1,\ldots,x_n$ in $E$, we have 
$$
\left(\sum_{i=1}^n\left\|T(x_i)\right\|^p\right)^{\frac{1}{p}}\leq C \sup_{x^*\in B_{E^*}}\left(\sum_{i=1}^n\left|x^*(x_i)\right|^p\right)^{\frac{1}{p}}.
$$
The infimum of such constants $C$ is denoted by $\pi_p(T)$ and the linear space of all $p$-summing operators from $E$ into $F$ by $\Pi_p(E,F)$. 

The following result could be compared with \cite[Proposition 3.13]{DelPiSer-10} stated in the linear setting.

\begin{proposition}
Let $f\in\H^\infty(U,F)$ and $g\in\H_{\K}^\infty(U,F^*)$. Assume that $T_f\in\Pi_p(\G^\infty(U),F)$ with $p\in [1,\infty)$. Then $f^t\circ g\in\H^\infty_{\K_p}(U,\H^\infty(U))$ with $k_p^{\H^\infty}(f^t\circ g)\leq\pi_p(T_f)\left\|g\right\|_\infty$. 
\end{proposition}

\begin{proof}
By Theorem \ref{linear} and Proposition \ref{new}, $T_g\in\K(\G^\infty(U),F^*)$ with $||T_g||=\left\|g\right\|_\infty$. Consequently, by \cite[Proposition 3.13]{DelPiSer-10}, the linear operator $(T_f)^*\circ T_g\in\K_p(\G^\infty(U),\G^\infty(U)^*)$ with $k_p((T_f)^*\circ T_g)\leq \pi_p(T_f)||T_g||$. From the equality $f^t\circ T_g=J_U^{-1}\circ (T_f)^*\circ T_g$, we infer that $f^t\circ T_g\in\K_p(\G^\infty(U),\H^\infty(U))$ with $k_p(f^t\circ T_g)=k_p((T_f)^*\circ T_g)$ by the ideal property of $\K_p$. Applying Theorem \ref{linear}, there exists $h\in\H^\infty_{\K_p}(U,\H^\infty(U))$ with $k_p^{\H^\infty}(h)=k_p(T_h)$ such that $f^t\circ T_g=T_h$. Hence $f^t\circ g=h$ and thus $f^t\circ g\in\H^\infty_{\K_p}(U,\H^\infty(U))$ with 
$$
k_p^{\H^\infty}(f^t\circ g)=k_p(T_h)=k_p((T_f)^*\circ T_g)\leq \pi_p(T_f)\left\|T_g\right\|=\pi_p(T_f)\left\|g\right\|_\infty.
$$
\end{proof}

$p$-Compact operators were characterized as those operators whose adjoints factor through a subspace of $\ell_p$ \cite[Theorem 3.2]{SinKar-02}. We now obtain a similar factorization for the transpose of a holomorphic mapping with relatively $p$-compact range (compare also to \cite[Proposition 3.10]{DelPiSer-10}). 

\begin{corollary}
Let $p\in [1,\infty)$ and $f\in\H^\infty(U,F)$. The following conditions are equivalent:
\begin{enumerate}
\item $f\colon U\to F$ has relatively $p$-compact range. 
\item There exist a closed subspace $M\subseteq\ell_{p}$ and $R\in\QN_p(F^*,M)$, $S\in\L(M,\H^\infty(U))$ such that $f^t=S\circ R$. 
\end{enumerate}
\end{corollary}

\begin{proof}
$(i)\Rightarrow (ii)$: If $f\in\H^\infty_{\K_p}(U,F)$, then $T_f\in\K_p(\G^\infty(U),F)$ by Theorem \ref{linear}. By \cite[Proposition 3.10]{DelPiSer-10}, there exist a closed subspace $M\subseteq\ell_{p}$ and operators $R\in\QN_p(F^*,M)$ and $S_0\in\L(M,\G^\infty(U)^*)$ such that $(T_f)^*=S_0\circ R$. Taking $S=J_U^{-1}\circ S_0\in\L(M,\H^\infty(U))$, we have $f^t=S\circ R$.

$(ii)\Rightarrow (i)$: Assume that $f^t=S\circ R$ being $S$ and $R$ as in the statement. It follows that $(T_f)^*=J_U\circ f^t=J_U\circ S\circ R$, and thus $T_f\in\K_p(\G^\infty(U),F)$ by \cite[Proposition 3.10]{DelPiSer-10}. Hence $f\in\H^\infty_{\K_p}(U,F)$ by Theorem \ref{linear}
\end{proof}


\subsection{Inclusions}

We will study the inclusion relations of holomorphic mappings with relatively $p$-compact range between them and with other classes of bounded holomorphic mappings. 

Our first result follows immediately by applying Theorem \ref{linear} and the fact stated in \cite[Proposition 4.3]{SinKar-02} that $\K_p\subseteq\K_q$ whenever $1\leq p\leq q\leq\infty$.

\begin{corollary}\label{new1}
If $1\leq p\leq q<\infty$ and $f\in\H^\infty_{\K_p}(U,F)$, then $f\in\H^\infty_{\K_q}(U,F)$ and $k_q^{\H^\infty}(f)\leq k_p^{\H^\infty}(f)$.$\hfill\qed$
\end{corollary}

Let us recall that a mapping $f\in\H^\infty(U,F)$ has \textit{finite dimensional rank} if the linear hull of its range is a finite dimensional subspace of $F$. We denote by $\H^\infty_{\F}(U,F)$ the set of all finite-rank bounded holomorphic mappings from $U$ to $F$. In the light of Theorem \ref{ideal}, it is clear that $\H^\infty_{\F}(U,F)$ is a linear subspace of $\H^\infty_{\K_p}(U,F)$. 

In similarity with the linear case, it seems natural to introduce the following class of holomorphic mappings.

\begin{definition}
Let $p\in [1,\infty)$. A mapping $f\in\H^\infty(U,F)$ is said to be \textit{$p$-approximable} if there exists a sequence $(f_n)$ in $\H^\infty_{\F}(U,F)$ such that $k_p^{\H^\infty}(f_n-f)\to 0$ as $n\to\infty$. We denote by $\H^\infty_{\overline{\F}_p}(U,F)$ the space of all $p$-approximable holomorphic mappings from $U$ to $F$.
\end{definition}

\begin{proposition}
For $p\in [1,\infty)$, every $p$-approximable holomorphic mapping from $U$ to $F$ has relatively $p$-compact range.
\end{proposition}

\begin{proof}
Let $f\in\H^\infty_{\overline{\F}_p}(U,F)$. Hence there is a sequence $(f_n)$ in $\H^\infty_{\F}(U,F)$ such that $k_p^{\H^\infty}(f_n-f)\to 0$ as $n\to\infty$. Since $T_{f_n}\in\F(\G^\infty(U),F)$ by \cite[Proposition 3.1]{Muj-91}, $\F(\G^\infty(U),F)\subseteq\K_p(\G^\infty(U),F)$ by \cite[Theorem 4.2]{SinKar-02} and $k_p(T_{f_n}-T_f)=k_p(T_{f_n-f})=k_p^{\H^\infty}(f_n-f)$ for all $n\in\mathbb{N}$ by Theorems \ref{main-theo} and \ref{linear}, we deduce that $T_f\in\K_p(\G^\infty(U),F)$ by \cite[Theorem 4.2]{SinKar-02}, and so $f\in\H^\infty_{\K_p}(U,F)$ by Theorem \ref{linear}.
\end{proof}

Given $p\in [1,\infty)$, $\ell^{\mathrm{weak}}_{p}(E)$ denotes the Banach space of all weakly $p$-summable sequences $(x_n)$ in $E$, endowed with the norm
$$
\left\|(x_n)\right\|^{\mathrm{weak}}_p=\sup\left\{\left(\sum_{n=1}^\infty\left|f(x_n)\right|^p\right)^{1/p}\colon f\in B_{E^*}\right\}.
$$
Let us recall (see \cite{Per-69}) that an operator $T\in\L(E,F)$ is said to be \textit{right $p$-nuclear} if there are sequences $(x^*_n)\in\ell^{\weak}_{p^*}(E^*)$ and $(y_n)\in\ell_{p}(F)$ such that
$$
T(x)=\sum_{n=1}^\infty x^*_n(x)y_n\qquad (x\in E),
$$
and the series converges in $\L(E,F)$. The set of such operators, denoted $\Nu^p(E,F)$, is a Banach space with the norm
$$
\nu^p(T)=\inf\left\{\left\|(x^*_n)\right\|^\weak_{p^*}\left\|(y_n)\right\|_p\right\},
$$
where the infimum is taken over all representations of $T$ as above. 

A holomorphic variant of this class of operators can be introduced as follows.

\begin{definition} 
Given $p\in [1,\infty)$, a holomorphic mapping $f\colon U\to F$ is said to be \textit{right $p$-nuclear} if there exist sequences $(g_n)$ in $\ell^{\mathrm{weak}}_{p^*}(\H^\infty(U))$ and $(y_n)$ in $\ell_{p}(F)$ such that $f=\sum_{n=1}^\infty g_n\cdot y_n$ in $(\H^\infty(U,F),\left\|\cdot\right\|_\infty)$. We set
$$
\nu^{p{\H^\infty}}(f)=\inf\left\{\left\|(g_n)\right\|^{\mathrm{weak}}_{p^*}\left\|(y_n)\right\|_{p}\right\},
$$
with the infimum taken over all right $p$-nuclear holomorphic representations of $f$ as above. Let $\H^\infty_{\Nu^p}(U,F)$ denote the set of all right $p$-nuclear holomorphic mappings from $U$ into $F$.
\end{definition}

We now establish the relationships of a right $p$-nuclear holomorphic mapping $f\colon U\to F$ with its linearisation $T_f\colon\G^\infty(U)\to F$.

\begin{theorem}\label{nuclear}
Let $p\in [1,\infty)$ and $f\in\H^\infty(U,F)$. The following conditions are equivalent:
\begin{enumerate}
	\item $f\colon U\to F$ is right $p$-nuclear.
	\item $T_f\colon\G^\infty(U)\to F$ is a right $p$-nuclear operator.
\end{enumerate}
In this case, $\nu^{p\H^\infty}(f)=\nu^p(T_f)$. Furthermore, the mapping $f\mapsto T_f$ is an isometric isomorphism from $(\H^\infty_{\Nu^p}(U,F),\nu^{p\H^\infty})$ onto $(\Nu^p(\G^\infty(U),F),\nu^p)$.
\end{theorem}

\begin{proof}
$(i)\Rightarrow (ii)$: Assume that $f\in\H^\infty_{\Nu^p}(U,F)$ and let $\sum_{n\geq 1}g_n\cdot y_n$ be a right $p$-nuclear holomorphic representation of $f$. By Theorem \ref{main-theo}, there is a unique operator $T_f\in\L(\G^\infty(U,F)$ such that $T_f\circ g_U=f$. Similarly, for each $n\in\mathbb{N}$, there is a functional $T_{g_n}\in\G^\infty(U)^*$ with $||T_{g_n}||=\left\|g_n\right\|_\infty$ such that $T_{g_n}\circ g_U=g_n$. Notice that $\sum_{n=1}^{+\infty}T_{g_n}\cdot y_n\in\L(\G^\infty(U),F)$ since 
$$
\sum_{k=1}^m\left\|T_{g_k}\cdot y_k\right\|=\sum_{k=1}^m\left\|T_{g_k}\right\|\left\|y_k\right\|=\sum_{k=1}^m\left\|g_k\right\|_\infty\left\|y_k\right\|\leq\left\|(g_n)\right\|^{\mathrm{weak}}_{p^*}\left\|(y_n)\right\|_{p}
$$
for all $m\in\mathbb{N}$. We can write
$$
f=\sum_{n=1}^\infty g_n\cdot y_n=\sum_{n=1}^\infty (T_{g_n}\circ g_U)\cdot y_n=\left(\sum_{n=1}^\infty T_{g_n}\cdot y_n\right)\circ g_U
$$
in $(\H^\infty(U,F),\left\|\cdot\right\|_\infty)$. Hence $T_f=\sum_{n=1}^\infty T_{g_n}\cdot y_n$ by Theorem \ref{main-theo}, where $(T_{g_n})\in\ell^{\mathrm{weak}}_{p^*}(\G^\infty(U)^*)$ with $\left\|(T_{g_n})\right\|^{\mathrm{weak}}_{p^*}\leq\left\|(g_n)\right\|^{\mathrm{weak}}_{p^*}$. Therefore $T_f\in\Nu^p(\G^\infty(U),F)$ with $\nu^p(T_f)\leq\left\|(g_n)\right\|^{\mathrm{weak}}_{p^*} \left\|(y_n)\right\|_p$. Taking infimum over all right $p$-nuclear holomorphic representation of $f$, we deduce that $\nu^p(T_f)\leq\nu^{p\H^\infty}(f)$.

$(ii)\Rightarrow (i)$: Suppose that $T_f\in\Nu^p(\G^\infty(U),F)$ and let $\sum_{n\geq 1}\phi_n\cdot y_n$ be a right $p$-nuclear representation of $T_f$. By Theorem \ref{main-theo}, for each $n\in\mathbb{N}$, there is a $g_n\in\H^\infty(U)$ such that $J_U(g_n)=\phi_n$ with $\left\|g_n\right\|_\infty=||\phi_n||$. We have 
\begin{align*}
\left\|\left(f-\sum_{k=1}^ng_k\cdot y_k\right)(x)\right\|&=\left\|f(x)-\sum_{k=1}^ng_k(x)y_k\right\|=\left\|T_f(g_U(x))-\sum_{k=1}^n J_U(g_k)(g_U(x))y_k\right\|\\
&=\left\|\left(T_f-\sum_{k=1}^n\phi_k\cdot y_k\right)(g_U(x))\right\|\leq \left\|T_f-\sum_{k=1}^n\phi_k\cdot y_k\right\|\left\|g_U(x)\right\|\\
&=\left\|T_f-\sum_{k=1}^n\phi_k\cdot y_k\right\|
\end{align*}
for all $x\in U$ and $n\in\mathbb{N}$. Taking supremum over all $x\in U$, we obtain
$$
\left\|f-\sum_{k=1}^ng_k\cdot y_k\right\|_\infty\leq\left\|T_f-\sum_{k=1}^n\phi_k\cdot y_k\right\| 
$$
for all $n\in\mathbb{N}$. Hence $f=\sum_{n=1}^\infty g_n\cdot y_n$ in $(\H^\infty(U,F),\left\|\cdot\right\|_\infty)$, where $(g_n)\in\ell^{\mathrm{weak}}_{p^*}(\H^\infty(U))$ with $\left\|(g_n)\right\|^{\mathrm{weak}}_{p^*}\leq\left\|(\phi_n)\right\|^{\mathrm{weak}}_{p^*}$. So $f\in \H^\infty_{\Nu^p}(U,F)$ with $\nu^{p\H^\infty}(f)\leq\left\|(\phi_n)\right\|^{\mathrm{weak}}_{p^*}\left\|(y_n)\right\|_p$, and this implies that $\nu^{p\H^\infty}(f)\leq\nu^p(T_f)$.

The last assertion in the statement follows easily from what was proved above and from Theorem \ref{main-theo}.
\end{proof}

Combining Theorem \ref{nuclear}, firstly with \cite[Theorem 3.2]{AroBotPelRue-10}, and secondly with \cite[Corollary 2.5]{CabJimRuiSep-22}, we derive the following two results.

\begin{corollary}\label{messi-3-3}
Let $p\in [1,\infty)$ and $f\in\H^\infty(U,F)$. The following conditions are equivalent:
\begin{enumerate}
\item $f\colon U\to F$ is right $p$-nuclear. 
\item $f=T\circ g$ for some complex Banach space $G$, $g\in\H^\infty(U,G)$ and $T\in\Nu^p(G,F)$. 
\end{enumerate}
In this case, we have
$$
\nu^{p\H^\infty}(f)=\left\|f\right\|_{\Nu^p\circ\H^\infty}:=\inf\{\nu^p(T)\left\|g\right\|_\infty\},
$$
where the infimum is taken over all factorizations of $f$ as in $(ii)$, and this infimum is attained at $T_f\circ g_U$.  

Furthermore, the mapping $f\mapsto T_f$ is an isometric isomorphism from $(\Nu^p\circ\H^\infty(U,F),\left\|\cdot\right\|_{\Nu^p\circ\H^\infty})$ onto $(\Nu^p(\G^\infty(U),F),\nu^p)$. $\hfill\qed$
\end{corollary}

\begin{corollary}\label{ideal-nuclear}
For each $p\in[1,\infty)$, $(\H^\infty_{\Nu^p},\nu^{p\H^\infty})$ is a Banach bounded-holomorphic ideal. $\hfill\qed$

\end{corollary}

The following relation is readily obtained.

\begin{corollary}\label{new1-1}
Let $p\in [1,\infty)$ and $f\in\H^\infty_{\Nu^p}(U,F)$. Then $f\in\H^\infty_{\K_p}(U,F)$ with $k_p^{\H^\infty}(f)\leq\nu^{p\H^\infty}(f)$.
\end{corollary}

\begin{proof}
By Proposition \ref{nuclear}, we have $T_f\in\Nu^p(\G^\infty(U),F)$ with $\nu^p(T_f)=\nu^{p\H^\infty}(f)$. It follows that $T_f\in\K_p(\G^\infty(U),F)$ with $k_p(T_f)\leq\nu^p(T_f)$ (see \cite[p. 295]{DelPiSer-10}). Hence $f\in\H^\infty_{\K_p}(U,F)$ with $k_p^{\H^\infty}(f)\leq\nu^{p\H^\infty}(f)$ by Theorem \ref{linear}.
\end{proof}

Given Banach spaces $E,F,G$, let us recall that a normed operator ideal $\I$ is \textit{surjective} if for every surjection $Q\in\L(G,E)$ and every $T\in\L(E,F)$, it follows from $T\circ Q\in\I(G,F)$ that $T\in\I(E,F)$ with $\left\|T\right\|_\I=\left\|T\circ Q\right\|_\I$. The smallest surjective ideal which contains $\I$, denoted by $\I^{\mathrm{sur}}$, is called the \textit{surjective hull of $\I$}. 

We now introduce the analogue concept in the holomorphic setting.

\begin{definition}
The \textit{surjective hull} of a bounded-holomorphic ideal $\I^{\H^\infty}$ is the smallest surjective ideal which contains $\I^{\H^\infty}$, and it is denoted by $(\I^{\H^\infty})^{\mathrm{sur}}$.
\end{definition}

We have seen that $\H^\infty_{\K_p}$ is a surjective bounded-holomorphic ideal which contains $\H^\infty_{\Nu^p}$, and therefore $(\H^\infty_{\Nu^p})^{\mathrm{sur}}\subseteq\H^\infty_{\K_p}$, but we do not know if both sets are equal as it happens in the linear case (see \cite[Proposition 3.11]{DelPiSer-10}) .


\textbf{Acknowledgements.} This research was partially supported by project UAL-FEDER grant UAL2020-FQM-B1858, by Junta de Andaluc\'{\i}a grants P20$\_$00255 and FQM194, and by grant PID2021-122126NB-C31 funded by MCIN/AEI/ 10.13039/501100011033 and by ``ERDF A way of making Europe''.


\end{document}